\documentclass[a4paper,11pt]{article}
\usepackage{stmaryrd}
\usepackage{bbm}
\usepackage[all,cmtip]{xy}
\usepackage{mathrsfs}
\usepackage{latexsym,amsfonts,amssymb,amsmath,amscd,amscd,amsthm,amsxtra}
\usepackage[dvips]{graphicx}
\usepackage[utf8]{inputenc}
\usepackage[T1]{fontenc}
\usepackage{lmodern}
\usepackage{nicefrac,mathtools}
\usepackage{microtype}
\usepackage{enumitem}
\usepackage{indentfirst}
\usepackage{cite}
\usepackage{kpfonts}
\usepackage{geometry}
\usepackage{ifpdf}
\ifpdf
\usepackage[colorlinks=true,linkcolor=blue,citecolor=red, final,backref=page,hyperindex]{hyperref}
\else
\usepackage[colorlinks,final,backref=page,hyperindex,hypertex]{hyperref}
\fi

\usepackage{enumitem}  
\usepackage{calc}
\setlist{labelindent=1pt,itemsep=.5em}
\setlist[itemize]{leftmargin=1.2cm}
\setlist[enumerate]{itemindent=0em,leftmargin=1.2cm}

\allowdisplaybreaks

\makeatletter
\newcommand{\subjclass}[2][2020]{%
  \let\@oldtitle\@title%
  \gdef\@title{\@oldtitle\footnotetext{#1 \emph{Mathematics subject classification}: #2}}%
}
\newcommand{\keywords}[1]{%
  \let\@@oldtitle\@title%
  \gdef\@title{\@@oldtitle\footnotetext{\emph{Keywords}: #1}}%
}
\makeatother

\allowdisplaybreaks
\newtheorem{defn}{Definition}[section]
\newtheorem{thm}[defn]{Theorem}
\newtheorem{lem}[defn]{Lemma}
\newtheorem{prop}[defn]{Proposition}
\newtheorem{cor}[defn]{Corollary}
\newtheorem{eg}[defn]{Example}
\newtheorem{re}[defn]{Remark}
\newcommand{\bdefn}{\begin{defn}}
\newcommand{\edefn}{\end{defn}}
\newcommand{\bthm}{\begin{thm}}
\newcommand{\ethm}{\end{thm}}
\newcommand{\blem}{\begin{lem}}
\newcommand{\elem}{\end{lem}}
\newcommand{\bprop}{\begin{prop}}
\newcommand{\eprop}{\end{prop}}
\newcommand{\bcor}{\begin{cor}}
\newcommand{\ecor}{\end{cor}}
\newcommand{\beg}{\begin{eg}}
\newcommand{\eeg}{\end{eg}}
\newcommand{\bre}{\begin{re}}
\newcommand{\ere}{\end{re}}
\newcommand{\bpf}{\begin{proof}}
\newcommand{\epf}{\end{proof}}

\newcommand{\benu}{\begin{enumerate}}
\newcommand{\eenu}{\end{enumerate}}
\newcommand{\bc}{\begin{center}}
\newcommand{\ec}{\end{center}}
\newcommand{\bea}{\begin{eqnarray}}
\newcommand{\eea}{\end{eqnarray}}
\newcommand{\ba}{\begin{align*}}
\newcommand{\ea}{\end{align*}}
\newcommand{\Bea}{\begin{eqnarray*}}
\newcommand{\Eea}{\end{eqnarray*}}
\newcommand{\beq}{\begin{equation}}
\newcommand{\eeq}{\end{equation}}
\newcommand{\Beq}{\begin{equation*}}
\newcommand{\Eeq}{\end{equation*}}
\newcommand{\bspl}{\begin{split}}
\newcommand{\espl}{\end{split}}

\numberwithin{equation}{section}

\begin{document}
\date{}
\title{\bf Pseudo-Euclidean Hom-alternative   superalgebras and Hom-post-alternative superalgebras}
\author{ S. Mabrouk,  O. Ncib, S. Sendi and S. Silvestrov}

\author{
 S. Mabrouk $^{2}$\footnote{E-mail: mabrouksami00@yahoo.fr }\ ,
O. Ncib $^{2}$\footnote{E-mail: othmenncib@yahoo.fr}\ , S. Sendi $^{1}$\footnote{E-mail: sihemsendi995@gmail.com} and S. Silvestrov $^{3}$\footnote{E-mail: sergei.silvestrov@mdu.se (Corresponding author)} 
}

\date{
$^{1}${\small University of Sfax, Faculty of Sciences Sfax,  BP 1171, 3038 Sfax, Tunisia} \\
$^{2}${\small  University of Gafsa, Faculty of Sciences Gafsa, 2112 Gafsa, Tunisia}\\{\small $^3$ M\"{a}lardalen University,
Division of Mathematics and Physics,}\\
{\small \hspace{1.5 cm}
School of Education, Culture and Communication, }\\ 
{\small \hspace{1.5 cm}
Box 883, 72123 V\"{a}steras, Sweden.}\\
}
\maketitle
\begin{abstract}
The purpose of this paper is to study pseudo-Euclidean and symplectic Hom-alternative superalgebras and  discuss some of their proprieties and provide construction procedures. We also introduce the notion of Rota-Baxter operators of pseudo-Euclidean Hom-alternative superalgebras of any weight and Hom-post-alternative superalgebras. A Hom-post-alternative superalgebra consists of three operations such that some compatibility conditions are satisfied. We show that a weighted Rota-Baxter operator  induces a Hom-post-alternative superalgebra naturally. Conversely, a Hom-post-alternative superalgebra gives rise to a new Hom-alternative superalgebra. In particular, a Hom-pre-alternative superalgebra is naturally built via symplectic structures.
\end{abstract}


\tableofcontents
\section{Introduction}
Hom-algebras first appeared in 2003 in the work of Hartwig, Larsson and Silvestrov \cite{hls}, in their investigation of the extension to general twisted derivations of $q$-deformations of the Witt and Virasoro algebras. Motivated by the twisted Jacobi identities in these algebras discovered in their work, they introduced Hom-Lie algebras and more general quasi-hom-Lie algebras with Jacobi identity twisted by linear maps. In 2005, Larsson and Silvestrov introduced quasi-Lie and quasi-Leibniz algebras  \cite{LarssonSilvestrov2005:QuasiLiealgebras} and color quasi-Lie and color quasi-Leibniz algebras  \cite{LarssonSilvestrov2005:GradedquasiLiealgebras}
incorporating within the same framework Hom-Lie algebras and quasi-Hom-Lie algebras, the color Hom-Lie algebras and Hom-Lie superalgebras, the color quasi-Hom-Lie algebras and quasi-Hom-Lie superalgebras, as well as the quasi-Leibniz algebras, the color quasi-Leibniz algebras and super quasi-Leibniz algebras.
Hom-Lie admissible algebras have been considered first in 2006 in the work of Makhlouf and Silvestrov \cite{MakhoufSilvestrov:Prep2006JGLTA2008:homstructure}, where
the Hom-associative algebras and more general $G$-Hom-associative algebras including the Hom-Vinberg algebras (Hom-left symmetric algebras), Hom-pre-Lie algebras (Hom-right symmetric algebras), and some other new Hom-algebra structures have been introduced and shown to be Hom-Lie admissible, in the sense that the operation of commutator as new product in these Hom-algebras structures yields Hom-Lie algebras. Furthermore, in \cite{MakhoufSilvestrov:Prep2006JGLTA2008:homstructure}, flexible Hom-algebras have been introduced and connections to Hom-algebra generalizations of derivations and of adjoint derivations maps have been considered, investigations of the classification problems for Hom-Lie algebras have been initiated with constriction of families of the low-dimensional Hom-Lie algebras. Various Hom-algebra structures have been considered (see for example \cite{cheng-su,yau2}). The Hom-algebras are widely explored in the last 15 years. For example, in \cite{amm-ej-makh,makh-sil2} the authors study cohomology and deformations of Hom-associative and Hom-Lie algebras. In particular, they generalize the classical Gerstenhaber bracket and Nijenhuis-Richardson bracket on the cochain complex of Hom-associative and Hom-Lie algebras. See also \cite{Chtioui1,Chtioui2,sheng-alg,yau2,yau3} and references therein for more on Hom-algebras.

An algebra that satisfies
\[
(xy)y = x(yy)
\]
is called a right alternative algebra.  If a right alternative algebra also satisfies the left alternative identity
\[
(xx)y = x(xy),
\]
then it is called an alternative algebra.  For example, the $8$-dimensional Cayley algebras are alternative algebras that are not associative  \cite{RD}.  Alternative algebras are closely related to other classes of non-associative algebras.  In fact, alternative algebras are Jordan-admissible and Maltsev-admissible \cite{bk,maltsev}.  Alternative algebras also satisfy the Moufang identities \cite{RD}. The Hom version of alternative algebra is studied in \cite{AM,yau3}, see also \cite{HamdediMakhlouf1,HamdediMakhlouf2}.

Alternative superalgebras \cite{EI} are the superalgebras whose associator is a super-alternating function. In particular, all associative superalgebras are alternative. Jordan superalgebras appeared in 1977–1980 \cite{VG}. In \cite{KFA} the authors introduce Hom-alternative, Hom-Malcev and Hom-Jordan superalgebras which are generalizations of alternative, Malcev and Jordan superalgebras respectively.

The purpose of this paper is to study Hom-type generalizations of alternative superalgebras, which are endowed with an even nondegenerate supersymmetric invariant bilinear form, and are called pseudo-Euclidean.

In 1960, G. Baxter \cite{Baxter} first introduced the notion of Rota–Baxter operators for associative algebras. The Rota-Baxter operators have several applications in probability \cite{Baxter}, combinatorics \cite{Cartier, Guo, Rota}, and quantum field theory \cite{Connes}. In the 1980s, the notion of Rota-Baxter operator of weight $0$ was introduced in terms of the classical Yang-Baxter equation for Lie algebras (see \cite{Survey-Guo} for more details).

Up to now, most of study on Rota-Baxter operators was on the associative algebras.
It was mentioned in \cite{Ebrahimi-Guo-Kreimer} that the Rota-Baxter operator and relation can be extended to Lie algebras or pre-Lie algebras. In the case of Lie algebras, when the
weight $\lambda=0$, the Rota-Baxter relation is just the operator form of the classical Yang-Baxter equation, and when the weight $\lambda = 1$, it corresponds to the operator form of the modified
classical Yang-Baxter equation \cite{Ebrahimi-Guo-Kreimer}. 

This paper is organized  as follows. In Section \ref{Prel}, we review the basic definitions and properties related to Hom-alternative superalgebras and give some construction. Section \ref{PEHomAlt} is devoted to the concept of
pseudo-Euclidean Hom-alternative superalgebras and  symplectic Hom-alternative superalgebra and the construction of symplectic  Hom-alternative superalgebras from pseudo-Euclidean symplectic Hom-Alternative superalgebras and an antisymmetric superderivation.
In  Section \ref{RotaBaxter}, we study the notion of weighted Rota-Baxter operators on Hom-alternative superalgebras and construct a Hom-post-alternative algebra consisting of three operations such that some compatibility conditions are satisfied for a given weighted Rota-Baxter operator of a Hom-alternative superalgebra. Conversely, a Hom-post-alternative superalgebra gives rise to a new Hom-alternative superalgebra. In particular, we show that a  Hom-pre-alternative superalgebra is naturally built via symplectic structures.

\textbf{Conventions and notations:}
\begin{enumerate}
    \item 
Throughout the paper, $\mathbb{K}$ denotes an algebraically closed field of characteristic $0$, and all vector spaces are over $\mathbb{K}$ and finite-dimensional.
\item We refer to the standard one-to-one correspondence between linear maps
$F : V_1\otimes\dots\otimes V_n \rightarrow W$ and multilinear maps
$F : V_1\times\dots\times V_n \rightarrow W$ given by
$F(v_1,\dots,v_n)=F(v_1\otimes\dots\otimes v_n)$, whenever the same notation is used for these maps.
\item A vector space $V$ is said to be a $\mathbb{Z}_2$-graded if we are given a family $(V_i)_{i\in\mathbb{Z}_2}$ of vector subspace of $V$ such that $V=V_0\oplus V_1.$
\item The symbol $|x|$ always implies that $x$ is a
$\mathbb{Z}_2$-homogeneous element and $|x|$ is the $\mathbb{Z}_2$-degree. We denote by $\mathcal{H(A)}$ the set of all homogeneous elements of $\mathcal{A}$ and $\mathcal{H}(\mathcal{A}^n)$ refers to the set of tuples with homogeneous elements. 
\item Let $End(\mathcal{A} )$ be the $\mathbb{Z}_{2}$-graded vector space of endomorphisms of a $\mathbb{Z}_{2}$-graded vector space $\mathcal{A}  = \mathcal{A} _{0} \oplus \mathcal{A} _{1}.$ The graded binary commutator $[f,g] = f \circ g - g \circ f$ induces the structure of Lie superalgebra in $End(\mathcal{A} )$.
\end{enumerate}
 \section{Basics on Hom-alternative superalgebras}\label{Prel}
In this section, we list some definitions and constructions on Hom-alternative superalgebras introduced in \cite{KFA}.
Recall that, a Hom-superalgebra is a triple $(\mathcal{A}, \mu, \alpha)$ in which $\mathcal{A}$ is a vector superspace,
 $\mu : \mathcal{A}\times \mathcal{A}\rightarrow \mathcal{A}$, is an even bilinear map and $\alpha : \mathcal{A}\rightarrow \mathcal{A}$ is an even linear map.
 \begin{enumerate}
     \item 
 The Hom-associator of $\mathcal{A}$ is the trilinear map $as_\mathcal{A}: \mathcal{A}\times \mathcal{A}\times \mathcal{A}\rightarrow \mathcal{A}$ defined
as $$as_\mathcal{A}=\mu\circ(\mu\otimes\alpha-\alpha\otimes\mu).$$
In terms of elements, the map $as_\mathcal{A}$ is given by
$$as_\mathcal{A}(x, y, z)=\mu(\mu(x, y), \alpha(z))-\mu(\alpha(x), \mu(y, z)),\;\forall x,y,z\in\mathcal{H}(\mathcal{A}).$$
\item An even linear map $f : (\mathcal{A}, \mu, \alpha) \rightarrow (\mathcal{A}', \mu', \alpha')$ is said to be a weak morphism of Hom-superalgebras if
$$f\circ\mu=\mu\circ(f\otimes f).$$
If in addition $f\circ\alpha=\alpha'\circ f$, $f$ is called a morphism of Hom-superalgebras.
\item A Hom-superalgebra $(\mathcal{A}, \mu, \alpha)$ is said to be  multiplicative if $\alpha\circ\mu=\mu\circ\alpha^{\otimes 2}$ .
\item A Hom-superalgebra $(\mathcal{A}, \mu, \alpha)$ is said to be regular (involutive) if $\alpha$ is bijective   ($\alpha^{2}=id$). It is easy to show that any involutive  Hom-superalgebra is regular.
 \end{enumerate}
 Throughout this work, all Hom-superalgebras are considered multiplicative.  
\begin{defn}
A Hom-associative superalgebra is a triple $( \mathcal{A}, \mu,
\alpha)$ consisting of $\mathbb{Z}_2$-graded vector space  $\mathcal{A}$, an even bilinear map $\mu : \mathcal{A}\times \mathcal{A} \rightarrow \mathcal{A}$ 
and an even linear map  $\alpha:\mathcal{A} \rightarrow \mathcal{A}$ satisfying for all $x, y, z\in \mathcal{H}(\mathcal{A})$
\begin{equation}
as_\mathcal{A}(x, y, z)=0.\label{as}
\end{equation}
\end{defn}
\begin{defn}[\hspace{-0,1mm}\cite{KFA}] \label{alt}
 A triple  $(\mathcal{A},\cdot,\alpha)$ is called a left Hom-alternative superalgebra if for all $x,y,z\in\mathcal{H}(\mathcal{A})$,  
\begin{equation}\label{LeftHomAlt}
as_\mathcal{A}(x, y, z)+(-1)^{|x||y|}as_\mathcal{A}(y, x, z)=0,\end{equation}
and respectively, right Hom-alternative superalgebra,  
if for all $x,y,z\in\mathcal{H}(\mathcal{A})$,
\begin{equation}\label{RightHomAlt}as_\mathcal{A}(x, y, z)+(-1)^{|y||z|}as_\mathcal{A}(x, z, y)=0.\end{equation}
\end{defn}
A Hom-alternative superalgebra is one which is both left and right Hom-alternative superalgebra. Of course, for $\alpha=Id$, one gets the definition of alternative superalgebras ( see  \cite{EI}) given by, for all $x,y,z\in \mathcal{H}(\mathcal{A})$, \begin{equation}\label{AltSupConditions}
as(x, y, z)+(-1)^{|x||y|}as(y, x, z)=0\quad\text{and}\quad as(x, y, z)+(-1)^{|y||z|}as(x, z, y)=0.    
\end{equation}Where 
$$as(x, y, z)=\mu(\mu(x, y), z)-\mu(x, \mu(y, z)).$$
\begin{re}
Any Hom-associative superalgebra is Hom-alternative, but a
 Hom-alternative superalgebra are not far removed from Hom-associative superalgebra.
\end{re}
\begin{defn}
A Hom-flexible superalgebra  is a triple 
$( \mathcal{A}, \cdot, \alpha)$ consisting of $\mathbb{Z}_2$-graded vector space  $\mathcal{A}$, an even bilinear map $\cdot : \mathcal{A}\times \mathcal{A} \rightarrow \mathcal{A}$ 
and an even linear map  $\alpha:\mathcal{A} \rightarrow \mathcal{A}$ satisfying for all $x, y, z\in \mathcal{H}(\mathcal{A})$,
\begin{equation}
as_\mathcal{A}(x, y, z)+(-1)^{|x||y|+|x||z|+|y||z|}as_\mathcal{A}(z, y,x)=0.\label{HomFlixible}
\end{equation}
\end{defn}
\begin{lem}
    \begin{enumerate}[label=\upshape{\arabic*.},left=5pt]
        \item Any Hom-alternative superalgebra is Hom-flexible superalgebra.
        \item If $(\mathcal{A},\cdot,\alpha)$ be a Hom-alternative superalgebra. Then we have
        \begin{equation}
as_\mathcal{A}(x, y, z)=(-1)^{|x|(|y|+|z|)}as_\mathcal{A}(y,z,x).\label{Alter2}
\end{equation}
    \end{enumerate}
\end{lem}
\begin{prop}\label{OpAlt}
Let  $(\mathcal{A},\cdot,\alpha)$ be a Hom-alternative superalgebra. Then $\mathcal{A}^{op}:=(\mathcal{A} ,\cdot^{op},\alpha)$, where 
$\cdot^{op} : \mathcal{A} \times \mathcal{A}  \rightarrow \mathcal{A} $ an even bilinear map such that for all $x,y\in \mathcal{H}(\mathcal{A} )$, $$x\cdot^{op} y=-(-1)^{|x||y|}y\cdot x,$$
is a Hom-alternative superalgebra.
\end{prop}
\begin{proof}
The associator in the opposite Hom-superalgebra is given by
\begin{eqnarray*}
 as_{\mathcal{A}^{op}}(x,y,z)
&=&(x\cdot^{op} y)\cdot^{op}\alpha(z)-\alpha(x)\cdot^{op}(y\cdot^{op}z)\\
&=&-(-1)^{|x||y|}(y\cdot x)\cdot^{op}\alpha(z)+(-1)^{|y||z|}\alpha(x)\cdot^{op}(z\cdot y))\\
&=&(-1)^{|x||y|+|z|(|x|+|y|)}(\alpha(z)\cdot(y\cdot x))-(-1)^{|y||z|+|x|(|z|+|y|)}((z\cdot y)\cdot \alpha(x))\\
&=&-(-1)^{|y||z|+|x||z|+|y||x|}((z\cdot y)\cdot\alpha(x)-\alpha(z)\cdot(y\cdot x))\\
&=&-(-1)^{|y||z|+|x||z|+|y||x|}as_{\mathcal{A}}(z,y,x).
\end{eqnarray*}
By \eqref{LeftHomAlt} and \eqref{RightHomAlt}, for all $x,y,z \in \mathcal{H}(\mathcal{A})$, 
\begin{align*}
 as_{\mathcal{A}^{op}}(x,y,z) 
&=-(-1)^{|y||z|+|x||z|+|y||x|}as_{\mathcal{A}}(z,y,x)\\ 
&=-(-1)^{|x||y|+|x||z|}as_{\mathcal{A}}(y,z,x)
=-(-1)^{|y||z|}as_{\mathcal{A}^{op}}(x,z,y)
\\
as_{\mathcal{A}^{op}}(x,y,z)&=-(-1)^{|y||z|+|x||z|+|y||x|}as_{\mathcal{A}}(z,y,x)\\ 
& =-(-1)^{|y||z|+|x||z|}as_{\mathcal{A}}(z,x,y))
=-(-1)^{|y||x|}as_{\mathcal{A}^{op}}(y,x,z).
\end{align*}
Therefore  $\mathcal{A}^{op}$ is a Hom-alternative superalgebra.
\end{proof}
\begin{re}
If $\mathcal{A}$ is only left (right) Hom-alternative superalgebra, $\mathcal{A}^{op}$ is only right (left) Hom-alternative superalgebra.
The importance of the opposite Hom-superalgebra is that it leads to a notion of duality for Hom-alternative superalgebra.
\end{re}
\begin{thm}[\hspace{-0.1mm}\cite{KFA}]
\label{TwistYauHomAltAlg}
Let $(\mathcal{A},\cdot,\alpha)$ be a Hom-alternative superalgebra and  $\beta:\mathcal{A}\to\mathcal{A}$ be morphism of  Hom-alternative superalgebra. Then $(\mathcal{A},\cdot_\beta=\beta \cdot,\beta\alpha)$
is a Hom-alternative superalgebra.
\end{thm}
\begin{defn}
A Hom-alternative superalgebra $(\mathcal{A},\cdot,\alpha)$ is said to be a Hom-alternative superalgebra of  alternative-type, if there exists another multiplication $\cdot^{'}$ such that $(\mathcal{A},\cdot^{'})$ is an alternative superalgebra where $\alpha(x\cdot^{'}y)=x\cdot y$  for $x,y\in\mathcal{H}(\mathcal{A})$.
\end{defn}
\begin{prop}\label{HomAltToAlt}
Any  regular Hom-alternative superalgebra $(\mathcal{A},\cdot,\alpha)$ is a Hom-alternative superalgebra of  alternative-type, where   $x\cdot^{'} y=\alpha^{-1}(x)\cdot\alpha^{-1}(y)$.
\end{prop}
\begin{proof}
Using the bijectivity of $\alpha$ and the multiplicativity of $(\mathcal{A},\cdot,\alpha)$ (assumed throughout the article), one has, for all $x,y,z \in \mathcal{H}(\mathcal{A})$,
\begin{align*}
\alpha(x\cdot^{'}y)=\alpha(\alpha^{-1}(x)\cdot\alpha^{-1}(y)) =&\alpha(\alpha^{-1}(x))\cdot\alpha(\alpha^{-1}(y)))= x\cdot y,\\
  as(x, y, z)+(-1)^{|x||y|}as(y, x, z)
=&(x \cdot^{'} y)\cdot^{'} z-x\cdot^{'}(y\cdot^{'} z)+(-1)^{|x||y|}\big((y\cdot^{'} x)\cdot^{'} z-y\cdot^{'}(x\cdot^{'} z)\big)\\
=&\alpha^{-1}(\alpha^{-1}(x)\cdot \alpha^{-1}(y))\cdot \alpha^{-1}(z)-\alpha^{-1}(x)\cdot \alpha^{-1}(\alpha^{-1}(y)\cdot \alpha^{-1}(z))\\
&+(-1)^{|x||y|}(\alpha^{-1}(\alpha^{-1}(y)\cdot\alpha^{-1}(x))\cdot \alpha^{-1}(z) -\alpha^{-1}(y)\cdot\alpha^{-1}(\alpha^{-1}(x)\cdot \alpha^{-1}(z)))\\
=&(\alpha^{-2}(x)\cdot \alpha^{-2}(y))\cdot\alpha^{-1}(z) -\alpha^{-1}(x)\cdot\alpha^{-2}(y)\cdot \alpha^{-2}(z)))\\
&+(-1)^{|x||y|}((\alpha^{-2}(y)\cdot\alpha^{-2}(x))\cdot \alpha^{-1}(z) -\alpha^{-1}(y)\cdot(\alpha^{-2}(x)\cdot\alpha^{-2}(z)))\\
=&\alpha^{-2}\big((x \cdot y)\cdot\alpha(z) -\alpha(x)\cdot(y\cdot z)+(-1)^{|x||y|}((y\cdot x)\cdot\alpha(z) -\alpha(y)\cdot(x \cdot z))\big)\\
=&\alpha^{-2}(as_{\mathcal A}(x,y,z)+(-1)^{|x||y|}as_{\mathcal A}(y,x,z))\\
=&0
 \end{align*}
by the left Hom-alternative super-identity of $(\mathcal{A},\cdot,\alpha)$. 
The right alternative super-identity is proved similarly.
\end{proof}
\section{Pseudo-Euclidean and symplectic Hom-alternative superalgebras}\label{PEHomAlt}
 In this section, we give some important constructions of pseudo-Euclidean  Hom-alternative superalgebras which are
 a Hom $\mathbb{Z}_2$-graded  generalizaton of pseudo-Euclidean  alternative algebras given in \cite{HeddouBoulmane} (see also \cite{AlbuquerqueBenayadi, BenamorBenayadi} ). Moreover,  we introduce the concept of symplectic  Hom-alternative superalgebras and we explore the relationship between the symplectic and pseudo-Euclidean Hom-alternative superalgebras. 
 \subsection{Pseudo-Euclidean   Hom-alternative superalgebras}
\begin{defn}
 Let $(\mathcal{A},\cdot,\alpha)$ be a Hom-alternative superalgebra. A bilinear form $\Psi$ on $\mathcal{A}$ is
 \begin{enumerate}[label=\upshape{\arabic*.},left=5pt]
\item supersymmetric if $\Psi(x,y)= (-1)^{|x||y|}\Psi(y,x)$, $\forall x, y\in \mathcal{H}(\mathcal{A})$;
\item nondegenerate if $x,y\in \mathcal{A}$ satisfies $\Psi(x,y)=0$, $\forall  y\in \mathcal{H}(\mathcal{A})$, then $x=0$;
\item invariant if \begin{equation}\label{InvCondition}
\Psi(x\cdot y,z)=\Psi(x,y\cdot z),\ \quad\forall x,y,z\in \mathcal{H}(\mathcal{A}),\end{equation} \item $\phi$-invariant for a morphism map $\phi:\mathcal{A}\to\mathcal{A}$ if \begin{equation}\label{InvCondition}
\Psi(x\cdot y,\phi(z))=\Psi(\phi(x),y\cdot z),\ \quad\forall x,y,z\in \mathcal{H}(\mathcal{A}),    
\end{equation}
\item even if $\Psi(\mathcal{A}_{0},\mathcal{A}_{1})=\Psi(\mathcal{A}_{1},\mathcal{A}_{0})=0$,
\item odd if $\Psi(\mathcal{A}_{0},\mathcal{A}_{0})=\Psi(\mathcal{A}_{1},\mathcal{A}_{1})=0.$
\end{enumerate}
 \end{defn}
\begin{defn}
A pseudo-Euclidean Hom-alternative superalgebra ($\mathsf{PEHomAlt}$ superalgebra),   
is a quadruple $(\mathcal{A},\cdot,\alpha, \Psi)$ such that  $(\mathcal{A},\cdot ,\alpha)$ is a Hom-alternative superalgebra equipped with an invariant, supersymmetric, 
non-degenerate bilinear form $\Psi$ such that for all $x,y\in\mathcal{H}(\mathcal{A})$, 
\begin{equation}\label{AlphaSymm}
    \Psi(\alpha(x),\alpha(y))=\Psi(x,y).\end{equation}
 \end{defn}
 \begin{defn}
A $\phi$-pseudo-Euclidean Hom-alternative superalgebra 
($\phi$-$\mathsf{PEHomAlt}$ superalgebra) is a quadruple $(\mathcal{A},\cdot ,\alpha, \Psi,\phi)$ such that  $(\mathcal{A},\cdot ,\alpha)$ is a Hom-alternative superalgebra  equipped with an $\phi$-invariant supersymmetric, non-degenerate bilinear form $\Psi$ satisfying the condition \eqref{AlphaSymm}.
\end{defn}
\begin{re}
\begin{enumerate}[label=\upshape{\arabic*.},left=5pt]
\item 
One recovers $\mathsf{PEAlt}$ superalgebra when $\alpha = id$ in a $\mathsf{PEHomAlt}$ superalgebra. 
\item An $id_\mathcal{A}$-$\mathsf{PEHomAlt}$ superalgebra is a $\mathsf{PEHomAlt}$ superalgebra.\end{enumerate}
 \end{re}
 \begin{re}
Note that the invariant condition \eqref{AlphaSymm} is not the same as the one given in \textup{\cite{Benayadi-Makhlouf}}, where it was used the condition 
\begin{equation}\label{AlphaSymm1}
\Psi(\alpha(x),y)=\Psi(x,\alpha(y)),\quad x,y\in \mathcal{H}(\mathcal{A}).
\end{equation}
Note that the compatibility condition between $\alpha$ and the multiplication $\cdot$ is that $\alpha$ is an algebra
morphism. For the Hom-algebra $(\mathcal{A},\cdot,\alpha)$, the set of algebra automorphisms is a group, called the
automorphism group. Thus, it is natural to require that $\alpha$ and the bilinear form $\Psi$ satisfy a similar
compatibility condition. One can show that the set of linear maps that preserve the bilinear form $\Psi$ in the sense of \textup{\cite{Benayadi-Makhlouf}} is a group. Actually, it is the orthogonal group associated to the bilinear form. However, linear maps that preserve the bilinear form $\Psi$ in the sense of \eqref{AlphaSymm1} do not have such
properties.
\end{re}
 Let $(\mathcal{A},\cdot ,\alpha, \Psi)$ and $(\mathcal{A}',\cdot' ,\alpha', \Psi')$ two $\mathsf{PEHomAlt}$ superalgebras. A lineair map $f :\mathcal{A}\to\mathcal{A}'$ is a morphism of $\mathsf{PEHomAlt}$ superalgebras if it is a morphism of Hom-alternative superalgebras satisfying for all 
 $x,y\in\mathcal{H}(\mathcal{A})$, 
\begin{equation}\label{PseudoMorphsim}
     \Psi'(f(x),f(y))=\Psi(x,y).
\end{equation}     
\begin{thm}
Let $(\mathcal{A},\cdot,\alpha,\Psi)$ be a $\mathsf{PEHomAlt}$ superalgebra and  $\beta:\mathcal{A}\to\mathcal{A}$ be morphism of pseudo-Euclidean Hom-alternative superalgebra. Then $(\mathcal{A},\cdot_\beta=\beta \cdot,\beta\alpha, \Psi)$ 
is a $\beta$-$\mathsf{PEHomAlt}$ superalgebra.
\end{thm}
\begin{proof} According to Theorem \ref{TwistYauHomAltAlg}, we have  $(\mathcal{A},\cdot_\beta=\beta \cdot,\beta\alpha)$ is a Hom-alternative superalgebra. Let $x,y,z\in \mathcal{H}(\mathcal{A})$, it is easy to check that $\Psi$ is supersymmetric and, we have 
\begin{align*}
    \Psi(x\cdot_\beta y,\beta(z))&=\Psi(\beta(x)\cdot\beta( y),\beta(z))=\Psi(\beta(x),\beta( y)\cdot\beta(z))=
    \Psi(\beta(x),y\cdot_\beta z)
\end{align*}
On the other hand, for $x,y\in \mathcal{H}(\mathcal{A})$  we have $$\Psi(\beta\alpha(x),\beta\alpha(y))=\Psi(\alpha(x),\alpha(y))=\Psi(x,y).$$
Therefore, $(\mathcal{A},\cdot_\beta,\beta\alpha, \Psi)$ 
is a $\beta$-$\mathsf{PEHomAlt}$ superalgebra.
\end{proof}
\begin{prop}
Let $(\mathcal{A},\cdot,\alpha,\Psi)$ be a regular $\alpha$-$\mathsf{PEHomAlt}$ superalgebra. Then $(\mathcal{A} , \cdot^{'}, \Psi)$
is a pseudo-Euclidean alternative superalgebra.
\end{prop}

\begin{proof} 
By Proposition \ref{HomAltToAlt}, 
$(\mathcal{A}, \cdot^{'})$ is an alternative superalgebra.
The invariance of $\Psi$ holds since for all $x, y, z\in \mathcal{H}(\mathcal{A})$,
\begin{align*}
\Psi(x\cdot^{'} y,z)=&\Psi(\alpha^{-1}(x)\cdot\alpha^{-1}(y),z)=\Psi(\alpha^{-1}(x)\cdot\alpha^{-1}(y),\alpha(\alpha^{-1}(z)))\\
=&\Psi(\alpha(\alpha^{-1}(x)),\alpha^{-1}(y)\cdot\alpha^{-1}(z))=
\Psi(x,y\cdot^{'} z),
\end{align*}
which completes the proof.
\end{proof}
\begin{prop}\label{opAlt}
Let  $(\mathcal{A},\cdot,\alpha,\Psi,\phi)$ be a $\phi$-$\mathsf{PEHomAlt}$ superalgebra. Then $(\mathcal{A},\cdot^{op},\alpha, \Psi,\phi )$ is $\phi$-$\mathsf{PEHomAlt}$ superalgebra, where $\cdot^{op}$ is defined in Proposition \ref{OpAlt}. 
\end{prop}
\begin{proof}
By Proposition \ref{OpAlt}, $(\mathcal{A},\cdot^{op},\alpha)$ is a Hom-alternative superalgebra. It suffices to prove that $\Psi$ is
$\phi$-invariant under $\cdot^{op}$. Let $x, y, z\in\mathcal{H}(\mathcal{A})$. 
Then,
\begin{align*}
\Psi(x\cdot^{op}y,\phi(z))&=-(-1)^{|x||y|}\Psi(y\cdot x,\phi(z))\\&=-(-1)^{|x||y|+|x||z|+|y||z|}\Psi(\phi(z),y \cdot x)\\
&=-(-1)^{|x||y|+|x||z|+|y||z|}\Psi(y \cdot z,\phi(x))\\&=-(-1)^{|y||z|}\Psi(\phi(x),z \cdot y)
=\Psi(\phi(x),y \cdot^{op} z),
\end{align*}
which completes the proof.
\end{proof}
In \cite{JCQ} the authors introduce the notion of 
 Hom-Malcev superalgebra, which  is a triple $(\mathcal{A},[\cdot,\cdot],\alpha)$ consisting of a superspace $\mathcal{A}=\mathcal{A}_{0}\oplus\mathcal{A}_{1}$, an even bilinear map $[\cdot,\cdot]: \mathcal{A}\times \mathcal{A} \rightarrow \mathcal{A}$ and an even linear map $\alpha:  \mathcal{A}\rightarrow \mathcal{A}$, satisfying
 for all $x,y,z\in\mathcal{H}(\mathcal{A})$, 
\begin{align}
& [x, y]=-(-1)^{|x||y|}[y, x],
\\ 
& \begin{array}{l}
(-1)^{|y||z|}[\alpha([x, z]),\alpha([y, t]) \\ 
=[[[x, y],\alpha(z)],\alpha^{2}(t)] 
+(-1)^{|x|(|y|+|z|+|t|)}[[[y, z],\alpha(t)],\alpha^{2}(x)]\\ 
+(-1)^{(|x|+|y|)(|z|+|t|)}[[[z, t],\alpha(x)],\alpha^{2}(y)]
+(-1)^{|t|(|x|+|y|+|z|)}[[[t, x],\alpha(y)],\alpha^{2}(z)].
\end{array}
\end{align}
We will give a relationship between Hom-alternative superalgebra and Hom-Malcev superalgebra.

\begin{prop}[\hspace{-0.1mm}\cite{KFA}]
\label{malcev}\label{malcev} Let
$(\mathcal{A}, \cdot,\alpha)$ be a Hom-alternative superalgebra, then $(\mathcal{A},[\cdot,\cdot], \alpha)$ is a Hom-Malcev superalgebra, where 
for all $x,y\in\mathcal{H}(\mathcal{A})$, 
\begin{align}\label{identitymalcev}
  [x,y]=x\cdot y-(-1)^{|x||y|}y\cdot x. 
\end{align}
\end{prop}

Now, we introduce the notion of $\phi$-quadratic Hom-Malcev superalgebra analogous  to $\phi$-$\mathsf{PEHomAlt}$ superalgebra. Then, we show that, any  $\phi$-$\mathsf{PEHomAlt}$ superalgebra induces a $\phi$-quadratic Hom-Malcev superalgebra.
\begin{defn}
A $\phi$-quadratic Hom-Malcev superalgebra is a Hom-Malcev superalgebra $(\mathcal{A},[\cdot,\cdot], \alpha)$ endowed with an even bilinear supersymmetric non degenerate form $B:\mathcal{A}\times \mathcal{A} \rightarrow \mathbb{K}$ and and even linear map $\phi:\mathcal{A} \rightarrow \mathcal{A}$, satisfying for all  $x,y,z\in\mathcal{H}(\mathcal{A})$ the conditions \eqref{AlphaSymm} and 
\begin{equation*}\label{InvQuad}
   B([x,y],\phi(z))+(-1)^{|x||y|}B(\phi(y),[x,z])=0.
\end{equation*}
\end{defn}
\begin{prop}
 Let $(\mathcal{A},\cdot,\alpha,\Psi)$ be a $\phi$-$\mathsf{PEHomAlt}$ superalgebra. Then, tuple $(\mathcal{A},[\cdot,\cdot],\alpha,\Psi )$, where 
 for all $x,y\in\mathcal{H}(\mathcal{A})$,
 $$[x,y]=x\cdot y-(-1)^{|x||y|}y\cdot x, $$
 is a $\phi$-quadratic Hom-Malcev superalgebra.
\end{prop}
\begin{proof}
Since  $(\mathcal{A},\cdot,\alpha,\Psi)$ is a $\mathsf{PEHomAlt}$ superalgebra, 
\begin{align*}
\Psi([x,y],\phi(z))
=&\Psi(x\cdot y-(-1)^{|x||y|}y\cdot x,\phi(z))\\
=&\Psi(x\cdot y,\phi(z))-(-1)^{|x||y|}\Psi(y\cdot x,\phi(z))\\
=&\Psi(x,y\cdot \phi(z))-(-1)^{|x||y|+|z|(|x|+|y|)}\Psi(\phi(z),y\cdot x))\\
=&\Psi(\phi(x),y\cdot z)-(-1)^{|x||y|+|z|(|x|+|y|)}\Psi(z\cdot y,\phi(x))\\
=&\Psi(\phi(x),y\cdot z)-(-1)^{|z||y|}\Psi(\phi(x),z\cdot y))\\
=&\Psi(\phi(x),y\cdot z-(-1)^{|z||y|}z\cdot y)\\
=&\Psi(\phi(x),[y,z]).
\end{align*}
Then $(\mathcal{A},[\cdot,\cdot],\alpha,\Psi )$ is a $\phi$-quadratic Hom-Malcev superalgebra.
\end{proof}
 \subsection{Symplectic Hom-alternative superalgebras}

\begin{defn}\label{symplectic}
Let $(\mathcal{A},\cdot,\alpha)$ be a Hom-alternative superalgebra. A super-skew-symmetric, non-degenerate even bilinear form $\omega$ on $(\mathcal{A},\cdot,\alpha)$ is said to be closed if it  satisfies 
\begin{align}\label{symp}
(-1)^{|x||z|}\omega(\alpha(x),y\cdot z)+(-1)^{|y||x|}\omega(\alpha(y),z\cdot x)+(-1)^{|z||y|}\omega(\alpha(z),x\cdot y)=0.  
\end{align}
for all $x,y,z\in\mathcal{H}(\mathcal{A}).$
Then $\omega$ is said to be symplectic. A Hom-alternative superalgebra $(\mathcal{A},\cdot,\alpha)$ with a symplectic form is called a symplectic Hom-alternative superalgebra. 
\end{defn}
\begin{defn}
Let $(\mathcal{A},\cdot,\alpha)$ be a Hom-alternative superalgebra. An homogeneous linear map $\mathfrak{D} :\mathcal{A} \rightarrow \mathcal{A}$ is called $\alpha^k$-superderivation if it satisfies 
\begin{equation}\label{derivation}
  \mathfrak{D}(x\cdot y)=\mathfrak{D}(x)\cdot \alpha^k(y)+(-1)^{|x||\mathfrak{D}|}\alpha^k(x)\cdot \mathfrak{D}(y),
\end{equation}
for all $x, y\in\mathcal{H}(\mathcal{A})$.
The set of all homogeneous $\alpha^k$-superderivations is denoted by $\mathfrak{D}er_{\alpha^k}(\mathcal{A})$. An $\alpha^k$-superderivation is called superderivation for $k=0.$ Set $$\mathfrak{D}er(\mathcal{A})=\bigoplus_{k\geq 0}\mathfrak{D}er_{\alpha^k}(\mathcal{A}).$$
The space $\mathfrak{D}er(\mathcal{A})$ with the commutator is a Lie superalgebra. Let $(\mathcal{A},\cdot ,\alpha,\Psi,\phi)$ be a $\phi$-$\mathsf{PEHomAlt}$ superalgebra. Define $\mathfrak{D}er_a(\mathcal{A})$ as the subspace of $\mathfrak{D}er(\mathcal{A})$ consisting of the antisymmetric 
$\alpha^k$-superderivations satisfying $\phi\mathfrak{D}=\mathfrak{D}\phi$ and 
\begin{align}\label{antiDer}
    &\Psi(\mathfrak D(x),y)+(-1)^{|x||\mathfrak{D}|}\Psi(x,\mathfrak D(y))=0,\quad \forall x,y\in \mathcal{H}(\mathcal{A}).
\end{align}
\end{defn}
\begin{lem}
    With the above notions, $\mathfrak{D}er_a(\mathcal{A})$ is a $\mathbb{Z}_2$-graded subalgebra of $\mathfrak{D}er(\mathcal{A})$.
\end{lem}
\begin{proof}
 Let $\mathfrak D,\mathfrak D'\in\mathfrak{D}er_a(\mathcal{A})$. It is easy to show that $[\mathfrak D,\mathfrak D']\phi=\phi[\mathfrak D,\mathfrak D']$. On the other hand, for all $x,y\in\mathcal{H}(\mathcal{A})$, we have
 \begin{align*}
     \Psi([\mathfrak D,\mathfrak D'](x),y)&=\Psi(\mathfrak D\circ\mathfrak D'(x),y)-(-1)^{|\mathfrak{D}||\mathfrak{D}'|}\Psi(\mathfrak D'\circ\mathfrak D(x),y)\\&=-(-1)^{|\mathfrak{D}|(|\mathfrak{D}'|+|x|)}\Psi(\mathfrak D'(x),\mathfrak D(y))+(-1)^{|x||\mathfrak{D}'|}\Psi(\mathfrak D(x),\mathfrak D'(y))\\&=(-1)^{|\mathfrak{D}||\mathfrak{D}'|}(-1)^{|x|(|\mathfrak{D}|+|\mathfrak{D}'|)}\Psi(x,\mathfrak D'\circ\mathfrak D(y))-(-1)^{|x|(|\mathfrak{D}|+|\mathfrak{D}'|)}\Psi(x),\mathfrak D\circ\mathfrak D'(y))\\&=-(-1)^{|x|(|\mathfrak{D}|+|\mathfrak{D}'|)}\Psi(x,[\mathfrak D,\mathfrak D]'(y)).
 \end{align*}
 Thus $[\mathfrak D,\mathfrak D']\in\mathfrak{D}er_a(\mathcal{A})$. Therefore, $\mathfrak{D}er_a(\mathcal{A})$ is a $\mathbb{Z}_2$-graded subalgebra of $\mathfrak{D}er(\mathcal{A})$.
\end{proof}
\begin{lem}\label{HomPEtoSymp}
Let $(\mathcal{A},\cdot ,\alpha,\Psi)$ be a $\alpha$-$\mathsf{PEHomAlt}$ superalgebra and $\mathfrak{D}$ is an even antisymmetric superderivation. Then $(\mathcal{A},\cdot,\alpha,\omega)$ is a symplectic Hom-alternative superalgebra, where
$$\omega(x,y)=\Psi(\mathfrak{D}(x),y),\quad x,y\in \mathcal{H}(\mathcal{A}).$$
\end{lem}
\begin{proof}
First, we assume that there is an invertible derivation $\mathfrak{D}$ on $\mathcal{A}$ such that $\Psi(\mathfrak{D}(x),y)=-(-1)^{|x||\mathfrak{D}|}\Psi(x,\mathfrak{D}(y))$. For any $x,y\in \mathcal{H}(\mathcal{A})$, define $\omega(x,y)=\Psi(\mathfrak{D}(x),y)$. Thus $\omega$ is a non-degenerate and super-skew
symmetric bilinear form. Now, we show that it is a symplectic structure. For all $x,y,z\in \mathcal{H}(\mathcal{A})$, we deduce 
\begin{align*}
   &(-1)^{|x||z|}\omega(\alpha(x),y\cdot z)+(-1)^{|y||x|}\omega(\alpha(y),z\cdot x)+(-1)^{|z||y|}\omega(\alpha(z),x\cdot y)\\
  = & (-1)^{|x||z|}\Psi(\mathfrak{D}(\alpha(x)),y\cdot z)+(-1)^{|y||x|}\Psi(\mathfrak{D}(\alpha(y)),z\cdot x)+(-1)^{|z||y|}\Psi(\mathfrak{D}(\alpha(z)),x\cdot y)\\
   =& -(-1)^{|x||z|}\Psi(\alpha(x),\mathfrak{D}(y\cdot z))+(-1)^{|y||x|}\Psi(\alpha(\mathfrak{D}(y)),z\cdot x)+(-1)^{|z||y|}\Psi(\alpha(\mathfrak{D}(z)),x\cdot y)\\
   =& -(-1)^{|x||z|}\Psi(\alpha(x),\mathfrak{D}(y\cdot z))+(-1)^{|y||x|}\Psi(\mathfrak{D}(y)\cdot z,\alpha( x))+(-1)^{|z||x|}\Psi(x\cdot y,\alpha(\mathfrak{D}(z)))\\
   =& -(-1)^{|x||z|}\Psi(\alpha(x),\mathfrak{D}(y\cdot z))+(-1)^{|z||x|}\Psi(\alpha(x),\mathfrak{D}(y)\cdot z)+(-1)^{|z||x|}\Psi(\alpha(x),y\cdot \mathfrak{D}(z))=0.
   \end{align*}
which completes the proof.
\end{proof}
\begin{prop}
Let   $(\mathcal{A},\cdot,\alpha,\omega)$ be a symplectic Hom-alternative superalgebra. Then $(\mathcal{A}^{op},\cdot^{op},\alpha,\Psi,\omega)$ is a  symplectic Hom-alternative superalgebra.
\end{prop}
\begin{proof}
By Proposition \ref{OpAlt}, $(\mathcal{A}^{op},\cdot^{op},\alpha)$ is a Hom-alternative superalgebra. Since $\mathcal{A}$ is a symplectic Hom-alternative superalgebra, then $\omega$ is a non-degenerate and super-skew
symmetric bilinear form on $\mathcal{A}^{op}$. We now prove that $\omega$ is an symplectic structure on $\mathcal{A}^{op}$. For $x,y \in\mathcal{A}^{op}$, we have
\begin{align*}
   &(-1)^{|x||z|}\omega(\alpha(x),y\cdot^{op} z)+(-1)^{|y||x|}\omega(\alpha(y),z\cdot^{op} x)+(-1)^{|z||y|}\omega(\alpha(z),x\cdot^{op} y)\\
   =&-(-1)^{|y||z|}(-1)^{|x||z|}\omega(\alpha(x),z\cdot y)-(-1)^{|z||x|}(-1)^{|y||x|}\omega(\alpha(y),x\cdot z)\\ 
   &\hspace{6cm} -(-1)^{|x||y|}(-1)^{|z||y|}\omega(\alpha(z),y\cdot x)\\
    =&(-1)^{|y||z|}\omega(\alpha(z),x\cdot y)+(-1)^{|z||x|}\omega(\alpha(x),y\cdot z)+(-1)^{|x||y|}\omega(\alpha(y),z\cdot x)=0.
 \qedhere  \end{align*}
\end{proof}
\begin{defn}
Let $(\mathcal{A},[\cdot,\cdot],\alpha)$ be a Hom-Malcev superalgebra and $\omega$ be a super-skew-symmetric non-degenerate even bilinear form on $\mathcal{A}$ satisfying:
$$(-1)^{|x||z|}\omega(\alpha(x),[y,z])+(-1)^{|y||x|}\omega(\alpha(y),[z,x])+(-1)^{|z||y|}\omega(\alpha(z),[x,y])=0,$$
for all $x,y,z\in\mathcal{H}( \mathcal{A})$.
Then $(\mathcal{A},[\cdot,\cdot],\alpha,\omega)$ is called a symplectic Hom-Malcev superalgebra.
If $(\mathcal{A},[\cdot,\cdot],\alpha,B)$ is a quadratic Hom-Malcev superalgebra and $\omega$ is a symplectic structure on $\mathcal{A}$, then  $(\mathcal{A},[\cdot,\cdot],\alpha,B,\omega)$ is called a quadratic symplectic Hom-Malcev superalgebra.
\end{defn}
\begin{prop}
Let $(\mathcal{A},\cdot ,\alpha,\omega)$ be a  symplectic Hom-alternative superalgebra. Then $(\mathcal{A},[\cdot,\cdot],\alpha,\omega )$ is a  symplectic Hom-Malcev superalgebra, where $[\cdot,\cdot]$ is defined by \eqref{identitymalcev}.
\end{prop}
\begin{proof}
According to Proposition \ref{malcev} and Definition \ref{symplectic},  we obtain
\begin{align*}
  &(-1)^{|x||z|}\omega(\alpha(x),[y,z])+(-1)^{|y||x|}\omega(\alpha(y),[z,x])+(-1)^{|z||y|}\omega(\alpha(z),[x,y])\\=& (-1)^{|x||z|}\big(\omega(\alpha(x),y\cdot z)-(-1)^{|y||z|}\omega(\alpha(x),z\cdot y)\big)\\&+(-1)^{|y||x|}\big(\omega(\alpha(y),z\cdot x)-(-1)^{|z||x|}\omega(\alpha(y),x\cdot z)\big)\\&+(-1)^{|z||y|}\big(\omega(\alpha(z),x\cdot y)-(-1)^{|x||y|}\omega(\alpha(z),y\cdot x)\big)\\=&-(-1)^{|x||z|+|y||z|}\omega(\alpha(x),z\cdot y) -(-1)^{|y||x|+|z||x|}\omega(\alpha(y),x\cdot z)-(-1)^{|z||y|+|x||y|}\omega(\alpha(z),y\cdot x)\\=&-(-1)^{|x||z|+|y||z|+|x||y|}\Big((-1)^{|x||y|}\omega(\alpha(x),z\cdot y) +(-1)^{|y||z|}\omega(\alpha(y),x\cdot z)+(-1)^{|x||z|}\omega(\alpha(z),y\cdot x)\Big)\\
  =&0.
 \end{align*}
 Then $(\mathcal{A},[\cdot,\cdot],\alpha,\omega )$ is a  symplectic Hom-Malcev superalgebra.
\end{proof}

\section{Weighted Rota-Baxter operators Hom-alternative superalgebras }\label{RotaBaxter}
In this section, we introduce Hom-post-alternative superalgebras as a generalization of Hom-pre-alternative superalgebras (see \cite{IB}).  Therefore, Hom-post-alternative superalgebras can be viewed as the underlying algebraic
structures of weighted Rota-Baxter on Hom-alternative superalgebras. A Hom-post-alternative superalgebra also gives rise
to a new Hom-alternative superlgebra. Hom-post-alternative superalgebras are analogues for Hom-alternative superalgebras of Hom-tridendriform  superalgebras. 
\subsection{ Hom-post-alternative superalgebras}
\begin{defn} \label{post-hom}
A Hom-post-alternative superalgebra is a tuple  
$(\mathcal{A}, \prec, \succ,\cdot, \alpha)$ in which $\mathcal{A}$  is a supervector space,
 $\prec, \succ,\cdot : \mathcal{A}\times \mathcal{A}\rightarrow \mathcal{A}$ are even bilinear maps and $\alpha : \mathcal{A}\rightarrow \mathcal{A}$ an even linear map such that,
for any $x, y, z\in\mathcal{H}(\mathcal{A})$,
\begin{align}
&
(x \cdot y) \cdot \alpha(z)- \alpha(x) \cdot (y \cdot z) + (-1)^{|x||y|}(y \cdot x) \cdot \alpha(z)  - (-1)^{|x||y|}\alpha(y) \cdot(x \cdot z)=0,\label{Hom-postaltr1} \\
&
(x \cdot y)\cdot \alpha(z)- \alpha(x) \cdot (y\cdot z) +(-)^{|y||z|} (x\cdot z) \cdot \alpha(y) -(-1)^{|y||z|}\alpha(x)\cdot (z \cdot y)=0,\label{Hom-postaltr2} \\
&
(x \cdot y) \prec \alpha(z) -\alpha(x) \cdot (y \prec z) + (-1)^{|x||y|}(y \cdot x) \prec \alpha(z)  -(-1)^{|x||y|} \alpha(y) \cdot (x \prec z)=0,\label{Hom-postaltr3} \\
&
(x \succ y) \cdot \alpha(z)- \alpha(x) \succ (y \cdot z) + (-1)^{|y||z|}(x \succ z) \cdot \alpha(y)  - (-1)^{|y||z|}\alpha(x) \succ (z \cdot y)=0,\label{Hom-postaltr4} \\
&
(y \succ x) \cdot \alpha(z)- \alpha(x) \cdot (y \succ z)+ (-1)^{|x||y|}(x \prec y) \cdot \alpha(z)  - (-1)^{|x||y|}\alpha(y) \succ (x \cdot z)=0,\label{Hom-postaltr5} \\
&
(z \prec x)\cdot \alpha(y)- \alpha(z) \cdot (x \succ y)+(-1)^{|x||y|}(z\cdot y) \prec \alpha(x) -(-1)^{|x||y|} \alpha(z) \cdot (y \prec x)  =0, \label{Hom-postaltr6}\\
&
(x\succ y) \prec \alpha(z) - \alpha(x)\succ(y \prec z) + (-1)^{|x||y|}(y \prec x) \prec \alpha(z) - (-1)^{|x||y|}\alpha(y) \prec (x \bullet z) = 0,\label{Hom-postaltr7}\\
&
(x\succ y)\prec \alpha(z) - \alpha(x)\succ(y \prec z) +(-1)^{|y||z|} (x \bullet z)\succ\alpha(y) -(-1)^{|y||z|}\alpha(x)\succ(z\succ y) = 0,\label{Hom-postaltr8}\\
&
(x\bullet y) \succ \alpha(z) -\alpha(x)\succ(y \succ z) + (-1)^{|x||y|}(y \bullet x) \succ \alpha(z) - (-1)^{|x||y|}\alpha(y) \succ (x \succ z) = 0,\label{Hom-postaltr9}\\
&
(z\prec x)\prec \alpha(y) - \alpha(z)\prec(x \bullet y) +(-1)^{|x||y|} (z \prec y)\prec \alpha(x) - (-1)^{|x||y|}\alpha(z)\prec(y\bullet x) = 0.\label{Hom-postaltr10}
\end{align}
where $x\bullet y=x\succ y+x\prec y+x\cdot y$.
\end{defn}
\begin{re}
If the operation $"\cdot"$ in the Hom-post-alternative superalgebra  $(\mathcal{A}, \prec , \succ , \cdot,\alpha  )$  is trivial \textup{(}$x\cdot y=0$ for all $x,y\in \mathcal{H}(\mathcal{A})$\textup{)}, then we recover the  Hom-pre-alternative superalgebra introduced in \textup{\cite{IB}}. Then the identities \eqref{Hom-postaltr1}-\eqref{Hom-postaltr10} reduce in Hom-pre-alternative superalgebras case to
\begin{align}\label{Hom-preAlt1}
  &\mathfrak{ass}_m(x,y,z)  +(-1)^{|x||y|}\mathfrak{ass}_r(y,x,z)=0,\\
  \label{Hom-preAlt2}
  &\mathfrak{ass}_m(x,y,z)  +(-1)^{|z||y|}\mathfrak{ass}_l(x,z,y)=0,\\
  \label{Hom-preAlt3}
  &\mathfrak{ass}_l(x,y,z)  +(-1)^{|x||y|}\mathfrak{ass}_l(y,x,z)=0,\\
  \label{Hom-preAlt4}
  &\mathfrak{ass}_r(x,y,z)  +(-1)^{|z||y|}\mathfrak{ass}_r(x,z,y)=0,
\end{align}
where for all $x,y,z\in \mathcal{H}(\mathcal{A})$, and 
$x\bullet y=x\succ y+x\prec y$,
\begin{align*}
&\mathfrak{ass}_l(x,y,z)=(x\bullet y) \succ \alpha(z) -\alpha(x)\succ(y \succ z), \\&\mathfrak{ass}_m(x,y,z)=(x \succ y) \prec \alpha(z) - \alpha(x) \succ (y \prec z),\\
&\mathfrak{ass}_r(x,y,z)=
(x\prec y)\prec \alpha(z) - \alpha(x)\prec(y \bullet z).
\end{align*}
\end{re}
\begin{thm}\label{HomPstAltToHomAlt}
Let $(\mathcal{A}, \prec , \succ , \cdot,\alpha  )$ be a Hom-post-alternative superalgebra. Let us define new bilinear operation for all $x,y\in \mathcal{H}(\mathcal{A})$ by 
\begin{equation}
x\bullet y= x\prec y+ x\succ y+ x\cdot y.
\end{equation}
Then $(\mathcal{A},\bullet,\alpha)$ is a Hom-alternative superalgebra.
\end{thm}
\begin{proof}
 Let us prove the left Hom-alternative super-identity. For any $x,y,z\in\mathcal{H}(\mathcal{A})$
\begin{align*}
&as_{\mathcal{A}} (x, y, z)+(-1)^{|x||y|}as_{\mathcal{A}}(y, x, z)\\=&
(x\bullet y)\bullet \alpha(z)-\alpha(x)\bullet(y\bullet z)\\
&+(-1)^{|x||y|} (y\circ x)\circ \alpha(z)-(-1)^{|x||y|}\alpha(y)\bullet(x\bullet z)\\
= &(x\bullet y)\prec \alpha(z)+(x\bullet y)\succ \alpha(z)\\ 
&+(x\bullet y)\cdot \alpha(z) -\alpha(x)\prec(y\bullet z)\\
&-\alpha(x)\succ(y\bullet z)-\alpha(x)\cdot(y\bullet z)\\
 &+(-1)^{|x||y|}(y\bullet x)\prec \alpha(z)
 +(-1)^{|x||y|}(y\bullet x)\succ \alpha(z)\\ 
 &+(-1)^{|x||y|}(y\bullet x)\cdot \alpha(z)-(-1)^{|x||y|}\alpha(y)\prec(x\bullet z)\\
 &-(-1)^{|x||y|}\alpha(y)\succ(x\bullet z)-(-1)^{|x||y|}\alpha(y)\cdot(x\bullet z)\\
= &(x\prec y)\prec \alpha(z)+(x\succ y)\prec \alpha(z)\\
&+(x\cdot y)\prec \alpha(z)+(x\bullet y)\succ \alpha(z)\\
&+(x\prec y)\cdot \alpha(z)+(x\succ y)\cdot \alpha(z)\\
 &+(x\cdot y)\cdot \alpha(z)-\alpha(x)\prec(y\bullet z)\\
 &-\alpha(x)\succ(y\prec z)- \alpha(x)\succ(y\succ z)\\
 &-\alpha(x)\succ(y\cdot z) -\alpha(x)\cdot(y\prec z)\\
&- \alpha(x)\cdot(y\succ z)-\alpha(x)\cdot(y\cdot z)\\ 
&+(-1)^{|x||y|}(y\prec x)\prec \alpha(z)+(-1)^{|x||y|}(y\succ x)\prec \alpha(z)\\ 
&+(-1)^{|x||y|}(y\cdot x)\prec \alpha(z)+(-1)^{|x||y|}(y\bullet x)\succ \alpha(z)\\ 
&+(-1)^{|x||y|}(y\prec x)\cdot \alpha(z)+(-1)^{|x||y|}(y\succ x)\cdot \alpha(z) \\ 
&+(-1)^{|x||y|}(y\cdot x)\cdot \alpha(z)-(-1)^{|x||y|}\alpha(y)\prec(x\bullet z)\\ 
&-(-1)^{|x||y|}\alpha(y)\succ(x\prec z)-(-1)^{|x||y|} \alpha(y)\succ(x\succ z)\\ 
&-(-1)^{|x||y|}\alpha(y)\succ(x\cdot z) -(-1)^{|x||y|}\alpha(y)\cdot(x\prec z)\\ 
&-(-1)^{|x||y|} \alpha(y)\cdot(x\succ z)-(-1)^{|x||y|}\alpha(y)\cdot(x\cdot z) =0.
\end{align*}
The right Hom-alternative super-identity is proved analogously.
\end{proof}
\begin{cor}[\hspace{-0.1mm}\cite{IB}]
Let $(\mathcal{A}, \prec, \succ, \alpha)$ be a Hom-pre-alterna\-tive superalgebra. If the operation $\bullet:\mathcal{A}\times \mathcal{A}$ is defined for all $x, y\in \mathcal{H}(\mathcal{A})$ by  
$x\bullet y=x\prec y+x\succ y,$ 
then $(\mathcal{A}, \bullet, \alpha)$ is a Hom-alternative superalgebra. 
\end{cor}
The Hom-alternative superalgebra $(\mathcal{A}, \bullet, \alpha)$ is called the associated Hom-alternative superalgebra of 
$(\mathcal{A}, \prec, \succ, \alpha)$). We call $(\mathcal{A}, \prec, \succ, \alpha)$ a compatible Hom-pre-alternative superalgebra
structure on the Hom-alternative superalgebra.

\begin{thm}
Let $(\mathcal{A}, \prec, \succ, \cdot,\alpha)$ be a Hom-post-alternative superalgebra and  $\beta:\mathcal{A}\to\mathcal{A}$ be morphism on $\mathcal{A}$. Then $(\mathcal{A}, \prec_\beta=\beta\prec, \succ_\beta=\beta\succ,\cdot_{\beta}=\beta\cdot, \beta\alpha)$ is a Hom-post-alternative superalgebra.
\end{thm}
\subsection{Rota-Baxter operators on pseudo-Euclidean Hom-alternative superalgebras}
\begin{defn}
Let $(\mathcal{A},\cdot,\alpha)$ be a Hom-alternative superalgebra and let $\lambda \in \mathbb{K}$. If an even linear map $\mathcal{R}:\mathcal{A}\rightarrow \mathcal{A}$ commutes with $\alpha$ and satisfies for all $x,y\in \mathcal{H}(\mathcal{A})$, 
\begin{eqnarray*}
R(x)\cdot R(y)=R(R(x)\cdot y+x\cdot R(y)+\lambda x\cdot y),
\end{eqnarray*}
then it is called a Rota-Baxter operator of weight $\lambda$ and $(\mathcal{A},\cdot,\alpha, R)$ is called a Rota-Baxter Hom-alternative superalgebra of weight $\lambda$.
\end{defn}

\begin{prop}
Let $(\mathcal{A},\cdot,\alpha,R)$ be a Rota-Baxter   Hom-alternative superalgebra of weight $\lambda$. Define a  multiplication on $\mathcal{A}$ by $x\cdot_{\alpha^{n}} y=\alpha^n(x)\cdot\alpha^n(y)$, for all $x,y\in\mathcal{H}(\mathcal{A})$.
Then $(\mathcal{A},\cdot_{\alpha^{n}},\alpha^{n+1},R)$ is a Rota-Baxter Hom-alternative superalgebra of weight $\lambda$.
\end{prop}
\begin{proof}
According to Corollary \ref{twist}, we have $(\mathcal{A},\cdot_{\alpha^{n}},\alpha^{n+1})$ is a  Hom-alternative superalgebra. For $x,y \in\mathcal{H}(\mathcal{A})$,
\begin{align*}
R(x)\cdot_{\alpha^{n}} R(y)
=&\alpha^{n+1}(R(x))\cdot\alpha^{n+1}(R(y))\\
=&R(\alpha^{n+1}(x))\cdot R(\alpha^{n+1}(y))\\
=&R(R(\alpha^{n+1}(x))\cdot\alpha^{n+1}(y)+\alpha^{n+1}(x)\cdot R(\alpha^{n+1}(y))+\lambda \alpha^{n+1}(x)\cdot\alpha^{n+1}(y)\\
=&R(\alpha^{n+1}(R(x))\cdot\alpha^{n+1}(y)+\alpha^{n+1}(x)\cdot\alpha^{n+1}(R(y))+\lambda \alpha^{n+1}(x)\cdot\alpha^{n+1}(y)\\
=&R(R(x)\cdot_{\alpha^{n}} y+ x\cdot_{\alpha^{n}} R(y)+\lambda x\cdot_{\alpha^{n}} y).
\end{align*}
Then $(\mathcal{A},\cdot_{\alpha^{n}},\alpha^{n+1}, R)$ is a Rota-Baxter Hom-alternative superalgebra of weight $\lambda$.
\end{proof}

\begin{prop}\label{HomAltToHomPstAlt}
 Let $(\mathcal{A},\cdot, \alpha)$ be a Hom-alternative superalgebra and $R: \mathcal{A}\rightarrow \mathcal{A}$ a Rota-Baxter operator of weight $\lambda$ on $\mathcal{A}$. Then
$(\mathcal{A}, \prec, \succ,\circ ,\alpha)$ is a Hom-post-alternative superalgebra, where for any $x, y \in \mathcal{H}(\mathcal{A})$, 
$$x\prec y=x\cdot R(y)\quad\mbox{,}\quad x\succ y=R(x)\cdot y,\quad x\circ y=\lambda x\cdot y.$$
\end{prop}
\begin{proof}
Since $\mathcal{A}$ is a Hom-alternative superalgebra, \eqref{Hom-postaltr1}, \eqref{Hom-postaltr2}, \eqref{Hom-postaltr4}, \eqref{Hom-postaltr5} and \eqref{Hom-postaltr6} obviously hold. For any $x, y, z\in \mathcal{H}(\mathcal{A})$, 
{\small\begin{align*}
&(x\succ y) \prec \alpha(z) - \alpha(x)\succ(y \prec z) + (-1)^{|x||y|}(y \prec x) \prec \alpha(z) - (-1)^{|x||y|}\alpha(y) \prec (x \bullet z)\\
=&(R(x)\cdot y)\cdot R(\alpha(z)) - R(\alpha(x))\cdot(y \cdot R(z)) + (-1)^{|x||y|}(y \cdot R(x)) \cdot R(\alpha(z)) - (-1)^{|x||y|}\alpha(y) \cdot R(x \bullet z)\\
=&(R(x)\cdot y)\cdot\alpha(R(z)) - \alpha (R(x))\cdot(y \cdot R(z)) + (-1)^{|x||y|}(y \cdot R(x)) \cdot \alpha (R(z)) - (-1)^{|x||y|}\alpha(y) \cdot R(x \bullet z)\\
=&(R(x)\cdot y)\cdot\alpha(R(z)) - \alpha (R(x))\cdot(y \cdot R(z)) + (-1)^{|x||y|}(y \cdot R(x)) \cdot \alpha (R(z)) \\ 
& \hspace{5cm}- (-1)^{|x||y|}\alpha(y) \cdot R( x\prec z+ x\succ z+ x\circ z)\\
=&(R(x)\cdot y)\cdot\alpha(R(z)) - \alpha (R(x))\cdot(y \cdot R(z)) + (-1)^{|x||y|}(y \cdot R(x)) \cdot \alpha (R(z)) \\ 
& \hspace{5cm} - (-1)^{|x||y|}\alpha(y) \cdot R( x\cdot R(z)+ R(x)\cdot z+ \lambda x\cdot z)\\
=&(R(x)\cdot y)\cdot\alpha(R(z)) - \alpha (R(x))\cdot(y \cdot R(z)) + (-1)^{|x||y|}(y \cdot R(x)) \cdot \alpha (R(z)) - (-1)^{|x||y|}\alpha(y) \cdot (  R(x)\cdot R(z))\\
=&0.
\end{align*}}
So, \eqref{Hom-postaltr7} holds. Moreover, \eqref{Hom-postaltr8} holds. Indeed,
{\small\begin{align*}
&(x\succ y)\prec \alpha(z) - \alpha(x)\succ(y \prec z) +(-1)^{|y||z|} (x \bullet z)\succ\alpha(y) -(-1)^{|y||z|}\alpha(x)\succ(z\succ y)\\
=&(R(x)\cdot y)\cdot R(\alpha(z)) - R(\alpha(x))\cdot(y \cdot R(z)) + (-1)^{|y||z|}R(x\bullet z)\cdot \alpha(y) - (-1)^{|y||z|}R(\alpha(x)) \cdot (R(z) \cdot y)\\
=&(R(x)\cdot y)\cdot R(\alpha(z)) - R(\alpha(x))\cdot(y \cdot R(z)) + (-1)^{|y||z|}R( x\prec z+ x\succ z+ x\circ z)\cdot \alpha(y)\\
& \hspace{5cm}- (-1)^{|y||z|}\alpha(R(x)) \cdot (R(z) \cdot y)\\
=&(R(x)\cdot y)\cdot\alpha(R(z)) - \alpha (R(x))\cdot(y \cdot R(z)) + (-1)^{|y||z|}R( x\cdot R(z)+ R(x)\cdot z+ \lambda x\cdot z)\cdot \alpha(y)\\
& \hspace{5cm}- (-1)^{|y||z|}\alpha(R(x)) \cdot (R(z) \cdot y)\\
=&(R(x)\cdot y)\cdot\alpha(R(z)) - \alpha (R(x))\cdot(y \cdot R(z))+(-1)^{|y||z|}(R(x)\cdot R(z))\cdot \alpha(y)- (-1)^{|y||z|}\alpha(R(x)) \cdot (R(z) \cdot y)\\
=&0.
\end{align*}}
Now, we prove the next identity \eqref{Hom-postaltr9}, we have
{\small\begin{align*}
&(x\bullet y) \succ \alpha(z) -\alpha(x)\succ(y \succ z) + (-1)^{|x||y|}(y \bullet x) \succ \alpha(z) - (-1)^{|x||y|}\alpha(y) \succ (x \succ z) \\
=&R(x\bullet y)\cdot \alpha(z) - R(\alpha(x))\cdot(R(y) \cdot z) + (-1)^{|x||y|}R(y\bullet x)\cdot \alpha(z) - (-1)^{|x||y|}R(\alpha(y)) \cdot (R(x) \cdot z)\\
=&R(x\prec y+x\succ y+x\circ y)\cdot \alpha(z) - R(\alpha(x))\cdot(R(y) \cdot z) + (-1)^{|x||y|}R(y\prec x+y\succ x+y\circ x)\cdot \alpha(z)\\
& \hspace{5cm} - (-1)^{|x||y|}R(\alpha(y)) \cdot (R(x) \cdot z)\\
=&R(x\cdot R(y)+R(x)\cdot y+\lambda x\cdot y)\cdot \alpha(z) - \alpha(R(x))\cdot(R(y) \cdot z) + (-1)^{|x||y|}R(y\cdot R(x)+R(y)\cdot x+\lambda y\cdot x)\cdot \alpha(z)\\
& \hspace{5cm} - (-1)^{|x||y|}\alpha(R(y) \cdot (R(x) \cdot z)\\
=&(R(x)\cdot R(y))\cdot \alpha(z) - \alpha(R(x))\cdot(R(y) \cdot z) + (-1)^{|x||y|}(R(y)\cdot R(x))\cdot \alpha(z) - (-1)^{|x||y|}\alpha(R(y) \cdot (R(x) \cdot z)\\
=&0.
\end{align*}}
The last identity \eqref{Hom-postaltr10} holds. By the same way, we show that
{\small\begin{align*}
&(z\prec x)\prec \alpha(y) - \alpha(z)\prec(x \bullet y) +(-1)^{|x||y|} (z \prec y)\prec \alpha(x) - (-1)^{|x||y|}\alpha(z)\prec(y\bullet x) \\
=&(z\cdot R(x))\cdot R(\alpha(y)) - \alpha(z)\cdot R(x \bullet y) +(-1)^{|x||y|} (z \cdot R(y))\cdot R(\alpha(x)) - (-1)^{|x||y|}\alpha(z)\cdot R(y\bullet x) \\
=&(z\cdot R(x))\cdot \alpha(R(y)) - \alpha(z)\cdot R(x\prec y+x\succ y+x\circ y) +(-1)^{|x||y|} (z \cdot R(y))\cdot \alpha(R(x)) \\
& \hspace{5cm} - (-1)^{|x||y|}\alpha(z)\cdot R(y\prec x+y\succ x+y\circ x)\\
=&(z\cdot R(x))\cdot \alpha(R(y)) - \alpha(z)\cdot R(x\cdot R(y)+R(x)\cdot y+\lambda x\cdot y) +(-1)^{|x||y|} (z \cdot R(y))\cdot \alpha(R(x)) \\
& \hspace{5cm} - (-1)^{|x||y|}\alpha(z)\cdot R(y\cdot R(x)+R(y)\cdot x+\lambda y\cdot x)\\
=&(z\cdot R(x))\cdot \alpha(R(y)) - \alpha(z)\cdot (R(x)\cdot R(y)) +(-1)^{|x||y|} (z \cdot R(y))\cdot \alpha(R(x))- (-1)^{|x||y|}\alpha(z)\cdot (R(y)\cdot R(x))\\
=&0.
\end{align*}}
Therefore $(\mathcal{A}, \prec, \succ,\circ ,\alpha)$ is a Hom-post-alternative superalgebra.
\end{proof}
Combining Theorem \ref{HomPstAltToHomAlt} and Proposition \ref{HomAltToHomPstAlt}, one obtains the following construction. 
\begin{cor}
Let $(\mathcal{A},\cdot,\alpha,R)$ be a Rota-Baxter Hom-alternative superalgebra of weight $\lambda$ and $R$ an even linear map commuting with $\alpha$.
Then $(\mathcal{A},\circ,\alpha,R)$ is a Rota-Baxter Hom-alternative superalgebra, where the multiplication $\circ$ is defined for all $x,y\in\mathcal{H}(\mathcal{A})$ as
\begin{equation}
x\circ y=R(x)\cdot y+x\cdot R(y)+\lambda x\cdot y.
\end{equation}
\end{cor}
 \begin{defn}
 Let $(\mathcal{A}, \cdot, \alpha,\Psi)$ be a $\mathsf{PEHomAlt}$ superalgebra and $R$ be a Rota-Baxter operators of weight $\lambda$. The tuple $(\mathcal{A}, \cdot, \alpha,\Psi,R)$ is called a pseudo-Euclidean Rota-Baxter   Hom-alternative superalgebra of weight $\lambda$ if following compatibility condition holds:
\begin{equation}
    \label{ComRB}
\Psi(Rx,y)+\Psi(x,Ry)+\lambda\Psi(x,y)=0.\end{equation}
 \end{defn}
 \begin{prop}
 Let $(\mathcal{A}, \cdot, \alpha,\Psi,R)$ be  a $\alpha$-pseudo-Euclidean Rota-Baxter   Hom-alternative superalgebra of weight $0$ where $R$ is invertible. Then,  $(\mathcal{A}, \cdot, \alpha,\Psi_R)$ is a symplectic Hom-alternative superalgebra, where 
 for all $x,y \in \mathcal{H}(\mathcal{A})$,  
 \begin{equation}
     \label{SymRB}\Psi_R(x,y)=\Psi(R^{-1}(x),y)). 
 \end{equation}
 \end{prop}
 \begin{proof} Since $R$ is invertible and $\Psi$ is non-degenerate and supersymmetric bilinear form,  it is easy to check that $\Psi_R$ is a non-degenerate super-skew-symmetric bilinear form. For $x, y\in \mathcal{H}(\mathcal{A})$, we have \begin{align*}
     R^{-1}(x\cdot y)&= R^{-1}(R(R^{-1}(x))\cdot R(R^{-1}(y))\\&=R^{-1}(R(R(R^{-1}(x))\cdot R^{-1}(y)+R^{-1}(x)\cdot R(R^{-1}(y)))\\&=R(R^{-1}(x))\cdot R^{-1}(y)+R^{-1}(x)\cdot R(R^{-1}(y))\\&=R^{-1}(x)\cdot y+x\cdot R^{-1}(y).
 \end{align*}
 Therfore, $R^{-1}$ is a derivation of degree $0$ on Hom-alternative superalgebra $(\mathcal{A}, \cdot, \alpha,\Psi,R)$. On the other hand, by \eqref{ComRB}, we have $$\Psi(R^{-1}x,y)=\Psi(R^{-1}x,RR^{-1}y)=\Psi(RR^{-1}x,R^{-1}y)=\Psi(x,R^{-1}y).$$  Thus, by  Lemma  
\ref{HomPEtoSymp}, $\Psi_R$ is a symplectic  form on the  Hom-alternative superalgebra $(\mathcal{A}, \cdot, \alpha)$.
 Then,  $(\mathcal{A}, \cdot, \alpha,\Psi\circ(R^{-1}\otimes id_\mathcal{A}))$ is a  symplectic Hom-alternative superalgebra.
\end{proof}
 
Let $(\mathcal{A}, \cdot, \alpha,\omega)$ be regular symplectic Hom-alternative superalgebra. Define the bilinear multiplications $\prec\mathcal{A}\times\mathcal{A}\rightarrow \mathcal{A}$ and 
$\succ: \mathcal{A}\times\mathcal{A}\rightarrow \mathcal{A}$ on $\mathcal{A}$, for all $x,y,z\in \mathcal{H}(\mathcal{A})$   by
\begin{align}
     &\omega(x\prec y, \alpha^2(z))=\omega(x, \alpha^{-1}(y)\cdot z),\label{symp1} \\& \omega(x\succ y, \alpha^2(z))=(-1)^{|x|(|y|+|z|)}\omega(y, z\cdot \alpha^{-1}(x)),\label{symp2}
\end{align}
which is due to the assumption of regularity (invertible $\alpha$) is equivalent to
\begin{align*}
     &\omega(x\prec \alpha(y), \alpha^2(z))=\omega(x, y\cdot z), \\& \omega(\alpha(x)\succ y, \alpha^2(z))=(-1)^{|x|(|y|+|z|)}\omega(y, z\cdot x).
\end{align*}

 \begin{lem}\label{BulletToCdot}
 For $x, y,z \in \mathcal{H}(\mathcal{A})$, 
\begin{align}
     \omega(x\bullet y, z)=\omega(x\cdot y, z).
\end{align}
 \end{lem}
 \begin{proof}
For any $x,y,z \in  \mathcal{H}(\mathcal{A})$, we have
\begin{align*}
\omega(x\bullet y,z)&=\omega(x\prec y, z)+\omega(x\succ y, z)\\&=\omega(x\prec \alpha(\alpha^{-1}(y)),\alpha^{2}(\alpha^{-2}(z)))+\omega(\alpha(\alpha^{-1}(x))\succ y,\alpha^{2}(\alpha^{-2}(z)))\\&=\omega(x,\alpha^{-1}(y)\cdot \alpha^{-2}(z))+(-1)^{|x|(|y|+|z|)}\omega(y,\alpha^{-2}(z)\cdot \alpha^{-1}(x))\\&=\omega(\alpha(x),y\cdot \alpha^{-1}(z))+(-1)^{|x|(|y|+|z|)}\omega(\alpha(y),\alpha^{-1}(z)\cdot x)\\&\stackrel{\eqref{symp}}{=} -(-1)^{|z|(|x|+|y|)}\omega(z,x\cdot y)=\omega(x\cdot y,z).
\qedhere 
\end{align*}
 \end{proof}
 \begin{thm}\label{comp}
Under the above notations, there
exists a compatible Hom-pre-alternative superalgebra structure $(\prec, \succ)$ on $\mathcal{A}$ with respect $\alpha$ given  in \eqref{symp1} and \eqref{symp2}.
 \end{thm}
 \begin{proof}
Let $x, y, z, t \in \mathcal{H}(\mathcal{A})$,  by  \eqref{symp1} and \eqref{symp2}, we have
\begin{align*}
& \omega(\mathfrak{ass}_m(x,y,z),\alpha^{2}(t))=\omega((x\succ y)\prec \alpha(z)-\alpha(x)\succ (y\prec z),\alpha^{2}(t))\\&=\omega((x\succ y)\prec \alpha(z),\alpha^{2}(t))-\omega(\alpha(x)\succ (y\prec z),\alpha^{2}(t))\\&=\omega(x\succ y,z\cdot t)-(-1)^{|x|(|y|+|z|+|t|)}\omega(y\prec z,t\cdot x)\\&=\omega(\alpha(\alpha^{-1}(x))\succ y,\alpha^{2}(\alpha^{-2}(z\cdot t))-(-1)^{|x|(|y|+|z|+|t|)}\omega(y\prec \alpha(\alpha^{-1}(z)),\alpha^{2}(\alpha^{-2}(t\cdot x))\\&=(-1)^{|x|(|y|+|z|+|t|)}\big(\omega(y,\alpha^{-2}(z\cdot t)\cdot \alpha^{-1}(x))-\omega(y,\alpha^{-1}(z)\cdot \alpha^{-2}(t\cdot x)\big)\\
&=(-1)^{|x|(|y|+|z|+|t|)}\big(\omega(\alpha^{-2}(\alpha^{2}(y)),\alpha^{-2}(z\cdot t)\cdot \alpha^{-2}(\alpha(x)))\\ 
&\hspace{4cm} -\omega(\alpha^{-2}(\alpha^{2}(y)),\alpha^{-2}(\alpha(z))\cdot \alpha^{-2}(t\cdot x))\big)\\
&=(-1)^{|x|(|y|+|z|+|t|)}\big(\omega(\alpha^{2}(y)),(z\cdot t)\cdot \alpha(x))-\omega(\alpha^{2}(y),\alpha(z)\cdot (t\cdot x)\big)\\
&=(-1)^{|x|(|y|+|z|+|t|)}\omega(\alpha^{2}(y),as(z,t,x)).
\end{align*}
On the other hand, by Lemma \ref{BulletToCdot} and   \eqref{Alter2}, \eqref{symp1} and \eqref{symp2}, we have
\begin{align*}
& (-1)^{|x||y|}\big(\omega(\mathfrak{ass}_r(y,x,z),\alpha^{2}(t))\\=&(-1)^{|x||y|}\omega((y\prec x)\prec \alpha(z)-\alpha(y)\prec (x\bullet z),\alpha^{2}(t)\big)\\=&(-1)^{|x||y|}\big(\omega((y\prec x)\prec \alpha(z),\alpha^{2}(t))-\omega(\alpha(y)\prec (y\bullet z),\alpha^{2}(t)\big)\\=&(-1)^{|x||y|}\big(\omega(y\prec x, z\cdot t)-\omega(\alpha(y)\prec \alpha(\alpha^{-1}(x\bullet z)),\alpha^{2}(t)\big)\\=&(-1)^{|x||y|}\big(\omega(y\prec \alpha(\alpha^{-1}(x)), \alpha^{2}(\alpha^{-2}(z\cdot t)))-\omega(\alpha(y), \alpha^{-1}(x\bullet z)\cdot t\big)\\=&(-1)^{|x||y|}\big(\omega(y,\alpha^{-1}(x)\cdot \alpha^{-2}(z\cdot t))-\omega(\alpha(y),\alpha^{-1}(x\bullet z)\cdot t)\big)\\=&-(-1)^{|x||y|}\big(\omega(\alpha^{2}(y),(x\cdot z)\cdot \alpha(t)-\alpha(x)\cdot (z\cdot t)))\\
=&-(-1)^{|x||y|}\omega(\alpha^{2}(y),as(x,z,t)).
\end{align*}
Thus,
\begin{align*}
&\omega(\mathfrak{ass}_m(x,y,z),\alpha^{2}(t))+(-1)^{|x||y|}\big(\omega(\mathfrak{ass}_r(y,x,z),\alpha^{2}(t)\big)\\&
=(-1)^{|x|(|y|+|z|+|t|)}\omega(\alpha^{2}(y),as(z,t,x)-(-1)^{|x||y|}\omega(\alpha^{2}(y),as(x,z,t))\\&=
\omega(\alpha^{2}(y),(-1)^{|x|(|z|+|t|)}as(z,t,x)-as(x,z,t))\\&=0.
\end{align*}
This implies that $\omega(\mathfrak{ass}_m(x,y,z)+(-1)^{|x||y|}\mathfrak{ass}_r(y,x,z),\alpha^{2}(t))=0$, for all $x,y,z,t \in \mathcal{H}(\mathcal{A})$. Since $\omega$ is non degenerate
and $\alpha$ is bijective, we have
$$\mathfrak{ass}_m(x,y,z)+(-1)^{|x||y|}\mathfrak{ass}_r(y,x,z)=0.$$
Then,  the identity \eqref{Hom-preAlt1} holds. Now we check the identity \eqref{Hom-preAlt2} for any $x,y,z,t \in \mathcal{H}(\mathcal{A})$, we have $\omega(\mathfrak{ass}_m(x,y,z),\alpha^{2}(t))=(-1)^{|x|(|y|+|z|+|t|)}\omega(\alpha^{2}(y),as(z,t,x))$.
On the other hand, we have
\begin{align*}
& (-1)^{|z||y|}\omega(\mathfrak{ass}_l(x,z,y),\alpha^{2}(t))=(-1)^{|z||y|}\omega\big( (x\bullet z)\succ\alpha(y)-\alpha(x)\succ (z\succ y),\alpha^{2}(t)\big)\\
&=(-1)^{|z||y|}\omega\big( (x\bullet z)\succ\alpha(y),\alpha^{2}(t)-\omega(\alpha(x)\succ (z\succ y),\alpha^{2}(t))\big)\\
&=(-1)^{|z||y|}\omega( \alpha(\alpha^{-1}(x\bullet z))\succ\alpha(y),\alpha^{2}(t))-\omega(\alpha(x)\succ (z\succ y),\alpha^{2}(t))\big)\\
&=(-1)^{|z||y|}\big((-1)^{(|x|+|z|)(|y|+|t|)}\omega(\alpha(y),t\cdot \alpha^{-1}(x\bullet z))-(-1)^{|x|(|z|+|y|+|t|)}\omega(z\succ y,t\cdot x)\big)\\
&=(-1)^{|z||y|}\big((-1)^{(|x|+|z|)(|y|+|t|)}\omega(\alpha^{2}(y),\alpha(t)\cdot (x\bullet z))\\ 
&\hspace{3cm} -(-1)^{|x|(|z|+|y|+|t|)}(-1)^{|z|(|y|+|t|+|x|)}\omega( y,\alpha^{-2}(t\cdot x)\cdot\alpha^{-1}(z)\big)\\
&=(-1)^{|z||y|}\big((-1)^{(|x|+|z|)(|y|+|t|)}\omega(\alpha^{2}(y),\alpha(t)\cdot (x\bullet z))\\
&\hspace{3cm} -(-1)^{|x|(|z|+|y|+|t|)}(-1)^{|z|(|y|+|t|+|x|)}\omega( \alpha^{2}(y),(t\cdot x)\cdot\alpha(z)\big)\\
&=-(-1)^{|z||y|}\big((-1)^{(|x|+|z|)(|y|+|t|)}\omega(\alpha^{2}(y),(t\cdot x)\cdot\alpha(z)-\alpha(t)\cdot (x\bullet z))\big)\\
&=-(-1)^{|z||y|}\big((-1)^{(|x|+|z|)(|y|+|t|)}\omega(\alpha^{2}(y),as(t,x,z)\big)\\
&=-(-1)^{|x||y|}\omega(\alpha^{2}(y),as(x,z,t)).
\end{align*}
In the same way, we have $\mathfrak{ass}_m(x,y,z)+(-1)^{|y||z|}\mathfrak{ass}_l(x,z,y)=0.$ 
The proof of other identities is similar.
\end{proof}



\end{document}